\documentclass[10pt,a4paper]{article}  
\usepackage[final]{optional}
\usepackage[british,american]{babel}
\usepackage{bbm}
\usepackage{mathrsfs}

\usepackage{amsfonts, amsmath, wasysym}

\usepackage{amssymb,amsthm,
paralist
}

\usepackage{
latexsym,
nicefrac,
}

\usepackage[usenames]{color}

\usepackage{url}

\definecolor{darkgreen}{rgb}{0,0.5,0}
\definecolor{darkred}{rgb}{0.7,0,0}
\usepackage[colorlinks, 
citecolor=darkgreen, linkcolor=darkred
]{hyperref}
\usepackage{esint}
\usepackage{bibgerm}
\usepackage[normalem]{ulem}

\theoremstyle{plain}
\newtheorem{lemma}{Lemma}[section]
\newtheorem{thm}[lemma]{Theorem}

\theoremstyle{definition}
\newtheorem{defn}[lemma]{Definition}

\newtheorem{rmk}[lemma]{Remark}

\numberwithin{equation}{section}

\newcommand{\m}{\mathcal{M}}

\newcommand{\cb}{\mathcal{B}}

\newcommand{\ci}{\mathcal{I}}
\newcommand{\cl}{\mathcal{L}}

\newcommand{\cu}{\mathcal{U}}

\newcommand{\pl}[2]{{\frac{\partial #1}{\partial #2}}}

\newcommand{\ga}{\gamma}

\newcommand{\de}{\delta}
\newcommand{\om}{\omega}
\newcommand{\Om}{\Omega}
\newcommand{\ka}{\kappa}
\newcommand{\la}{\lambda}

\newcommand{\si}{\sigma}

\newcommand{\Si}{\Sigma}

\newcommand{\ep}{\varepsilon}
\newcommand{\Ups}{\Upsilon}

\newcommand{\R}{\ensuremath{{\mathbb R}}}
\newcommand{\N}{\ensuremath{{\mathbb N}}}

\newcommand{\downto}{\downarrow}
\newcommand{\upto}{\uparrow}

\newcommand{\grad}{\nabla}

\newcommand{\union}{\cup}

\DeclareMathOperator{\inj}{inj}

\newcommand{\norm}[1]{\Vert#1\Vert}  

\newcommand{\beq}{\begin{equation}}
\newcommand{\eeq}{\end{equation}}
\newcommand{\beqs}{\begin{equation}}
\newcommand{\eeqs}{\end{equation}}
\newcommand{\beqa}{\begin{equation}\begin{aligned}}
\newcommand{\eeqa}{\end{aligned}\end{equation}}
\newcommand{\beqas}{\begin{equation}\begin{aligned}}
\newcommand{\eeqas}{\end{aligned}\end{equation}}
\newcommand{\brmk}{\begin{rmk}}
\newcommand{\ermk}{\end{rmk}}
\newcommand{\partref}[1]{\hbox{(\csname @roman\endcsname{\ref{#1}})}}
\newcommand{\half}{\frac{1}{2}}

\newcommand{\lie}{\mathcal{L}}

\newcommand*\dist{\mathop{\mathrm{dist}}\nolimits}

\newcommand*\arsinh{\mathop{\mathrm{arsinh}}\nolimits}

\newcommand*\ddt{\frac{d}{dt}}

\newcommand{\pt}{\partial_t}
\newcommand{\M}{\ensuremath{{\mathcal M}}_{-1}}

\newcommand{\abs}[1]{\vert#1\vert}

\newcommand{\Col}{\mathcal{C}}

\newcommand{\thin}{\text{-thin}}
\newcommand{\thick}{\text{-thick}}

\title{{\sc
horizontal curves of hyperbolic metrics
}
\\ 
}
\author{ Melanie Rupflin and Peter M. Topping}
\date{\today}

\begin{document}
\maketitle

\begin{abstract}
We analyse the fine convergence properties of one parameter families of hyperbolic metrics that move always in a horizontal direction, i.e. orthogonal to the action of diffeomorphisms. Such families arise naturally in the study of general curves of metrics on surfaces,
and in one of the gradients flows for the harmonic map energy.
\end{abstract}

\section{Introduction}

It has been extensively studied how a general sequence of hyperbolic metrics on a fixed closed oriented surface $M$ can degenerate. In general, the length of the shortest closed geodesic will converge to zero and the surface will stretch an infinite amount as it develops a long thin collar around each short closed geodesic, as described precisely by the classical Collar Lemma that we recall in an appropriate form in the appendix (Lemma \ref{lemma:collar}). A differential geometric form of the Deligne-Mumford compactness theorem \ref{Mumford} tells us that \emph{modulo diffeomorphisms} the closed hyperbolic surfaces converge to a complete, possibly noncompact, hyperbolic surface with cusp ends, after passing to a subsequence.

In this paper, we are concerned with the convergence of a smooth one-parameter family of hyperbolic metrics $g(t)$, $t\in [0,T)$, as $t\upto T$. It is natural to consider such families that move orthogonally to the 
variations by diffeomorphisms, i.e. so-called horizontal curves, as we clarify now. Writing $\M$ for the space of hyperbolic metrics on $M$, i.e. the metrics of constant Gauss curvature $-1$, we recall 
(cf. \cite{tromba})
that each tangent space $T_g\M$ enjoys the $L^2$ decomposition 
$$
\{Re(\Psi):\Psi\text{ is a holomorphic quadratic differential on }(M,g)\}\oplus
\{\lie_X g:X\in\Gamma(TM)\},$$
which motivates the following standard definition.
\begin{defn}
\label{horiz_def}
Let $M$ be a smooth closed oriented surface of genus at least $2$, and 
let $g(t)$ be a smooth one-parameter family of  hyperbolic metrics on $M$ for $t$ within some interval $I\subset\R$.
We say that $g(t)$ is a \emph{horizontal curve} if for each $t\in I$, there exists a holomorphic quadratic differential $\Psi(t)$ such that $\pl{g}{t}=Re(\Psi)$. 
\end{defn}

Every smooth one-parameter family $\tilde g(t)$ of metrics on such $M$ has some horizontal curve $g(t)$ at its heart as we now explain.
First, we can write $\tilde g(t)$ uniquely as 
$\tilde g(t)=e^{v(t)} \hat g(t)$ for a smooth curve of hyperbolic metrics $\hat g(t)$ and a smooth one-parameter family of functions $v(t):M\to\R$. The resulting curve $\hat g(t)$ of hyperbolic metrics in turn can then be written uniquely as 
$\hat g(t)=f_t^* g(t)$ for a horizontal curve $g(t)$ 
and a smooth family of diffeomorphisms $f_t:M\to M$ with $f_0$ the identity, hence giving a natural decomposition 
\beq
\label{gen_curve_decomp}
\tilde g(t)=e^{v(t)}f_t^*g(t),
\eeq
into
a horizontal curve $g(t)$, conformal changes and pull-backs by diffeomorphisms. 

We note that while we may also decompose families of metrics on surfaces of lower genus, the horizontal part of such families is trivial to analyse as it is constant (for genus $0$) or described in terms of a curve that is contained in an explicit 2 dimensional manifold (for genus $1$).

A horizontal curve $g(t)$  can be projected down to a path in Teichm\"uller space whose
length with respect to the Weil-Petersson metric over a range $t\in [s,T)$ is given (up to normalisation) by 
\beq
\label{length_def}
\cl(s):=\int_s^T\|\pt g(t)\|_{L^2(M,g(t))}dt\in [0,\infty].
\eeq
Clearly, to have any hope of any reasonable convergence of a horizontal curve $g(t)$ as $t\upto T$, we must ask that its length is \emph{finite}, equivalently that $\cl(s)\downto 0$ as $s\upto T$.
By the incompleteness of the Weil-Petersson metric on
Teichm\"uller space on surfaces of genus at least $2$, the curve having finite length does not rule out degeneration of the metric, in contrast to the analogous situation on a torus.

Even for such  horizontal curves of finite length, however close $t<T$ has got to $T$ the surface will still have infinitely much stretching to do, in general.
This paper is dedicated to proving the following main convergence result for finite-length horizontal curves, without any modification by diffeomorphisms. The result extracts a smooth complete limit on a subset of $M$, and precisely describes where on this subset the metric $g(t)$ has essentially settled down to its limit, and where any infinite amount of remaining stretching must occur.

\begin{thm}
\label{convergence_of_g2}
Let $M$ be a closed oriented surface of genus $\ga\geq 2$,
and suppose $g(t)$ is a smooth horizontal curve in $\m_{-1}$, for $t\in [0,T)$,  
with finite length $\cl(0)<\infty$.
Then there exist a nonempty open subset $\cu\subset M$, 
whose complement has $\ka\in\{0,\ldots,3(\ga-1)\}$ connected components,
and
a complete hyperbolic metric $h$ on $\cu$ for which
$(\cu,h)$ is of finite volume, with cusp ends, and is conformally a  disjoint union of 
$m\in\{1,\ldots,2(\ga-1)\}$
closed Riemann surfaces (of genus strictly less than that of $M$ if $\cu$ is not the whole of $M$)
with a total of $2\ka$ punctures, such that
$$g(t)\to h$$
smoothly locally on $\cu$.
Moreover, defining a function $\ci:M\to[0,\infty)$ by
$$\ci(x)=\left\{
\begin{aligned}
& \inj_h(x) &\quad &  \text{on }\cu\\
& 0 & \quad &  \text{on }M\backslash \cu,
\end{aligned}
\right.
$$
we have $\inj_{g(t)}\to \ci$ uniformly on $M$ as $t\upto T$, and indeed that
\beq
\label{unif_conv}
\left\|[\inj_{g(t)}]^\half-\ci^\half\right\|_{C^0}\leq K_0\cl(t)\to 0\qquad\text{as }t\upto T
\eeq
where $K_0$ depends only on the genus of $M$ (and is determined in Lemma \ref{rootinjevol2}). 
Furthermore, for  any $k\in\N$ and 
$\de>0$, if we take $t_0\in [0,T)$ sufficiently large so that 
$(2K_0 \cl(t_0))^2< \de$, and any $t\in [t_0,T)$,
then
\beq
\label{h_est_lem_h}
\|g(t)-h\|_{C^k(\de\thick(\cu,h),h)}\leq C\de^{-\half}\cl(t),
\eeq
and for all $s\in [t_0,T)$ we have $\de\thick(M,g(s))\subset \cu$ and 
\beq
\label{h_est_lem}
\|g(t)-h\|_{C^k(\de\thick(M,g(s)),g(s))}\leq C\de^{-\half}\cl(t),
\eeq
where $C$ depends only on $k$ and the genus of $M$.
\end{thm}

Implicit above is the fact that because we are working on hyperbolic surfaces, 
the function 
$x\mapsto \inj_{g(t)}(x)$ is a continuous function for each $t\in [0,T)$, and thus the convergence of \eqref{unif_conv} is genuinely $C^0$ convergence of continuous functions to a continuous limit. (See also 
Remark \ref{si_rmk}.)
The definition of the $C^k$ norm is made precise in \eqref{Ck_notation2}.

Theorem \ref{convergence_of_g2} is particularly useful in the case that the horizontal curve degenerates, i.e. when $\liminf_{t\upto T}\inj_{g(t)}(M)=0$,
or equivalently when $\cu$ is not the whole of $M$.
In this case,
a significant aspect of Theorem \ref{convergence_of_g2} is that if $\de(t)$ converges to zero more slowly than $\cl(t)^2$ as $t\upto T$, in the sense that $\frac{\cl(t)^2}{\de(t)}\to 0$ as $t\upto T$, then $g(t)$ must have settled down to its limit on the $\de(t)$-thick part by time $t$ (the `solid' part). The infinite amount of subsequent stretching required to pinch a collar would then be going on purely on the $\de(t)$-thin part (the `liquid' part).

This type of control can be considered an unusual perspective on the very familiar concept of collar degeneration. Nevertheless, this perspective is fundamental when the hyperbolic surface $(M,g(t))$ serves as the domain on which we solve a PDE such as a geometric flow; in this case a departure from the usual  viewpoint modulo diffeomorphisms is necessary. 
A first instance of this occurs in our work on the Teichm\"uller harmonic map flow, which is a flow that takes a map $u$ from a surface $(M,g)$ to a general closed Riemannian manifold $(N,G)$ and flows both $u$ and $g$ in order to reduce the harmonic map energy $E(u,g)$ as quickly as possible, see \cite{RT} for details. The way the flow is set up, the metric $g$ always moves in a horizontal direction, and so the theory of this paper applies instantly.
Indeed, in \cite{global}, we use Theorem \ref{convergence_of_g2}
to describe the finite-time singularities of the flow for both the metric and map components. On the `solid' part of the domain, where the metric $g(t)$ has settled down near to its limit, the map $u$ will converge in a traditional sense. Meanwhile, the theory of this paper controls the `liquid' part, where all the stretching is yet to occur, in a sufficiently strong manner that the map looks like a harmonic map at every scale and from every viewpoint.

Because of the decomposition \eqref{gen_curve_decomp}
 the results of the present paper are applicable not only in situations where a curve of metrics 
moves purely in horizontal directions, but also for general curves of metrics on surfaces.   
 Some of the technology we develop here has already been applied in such a situation in \cite{buz_rup},
 where the theory of the present paper controls the horizontal part while different techniques control the conformal
deformations and modifications by diffeomorphisms.
A more recent application of this type can be found in \cite{lothar}.

Other results from this paper that are used elsewhere, particularly in \cite{global}, include our direct control on the injectivity radius (Lemma \ref{rootinjevol2}), our elliptic estimate for $\pt g$ (Lemma \ref{yaba_lemma}), which controls 
$|\pt{g}|_{C^k}(x)\leq C[\inj_{g(t)}(x)]^{-\half}\|\pt g\|_{L^2}$,
and our results relating the geometry at different times (Lemma \ref{lemma:Ck-horiz}).

\emph{Acknowledgements:} 
The second author was supported by 
EPSRC grant number EP/K00865X/1.

\section{Injectivity radius along horizontal curves}
\label{inj_rad_sect}

In this section, we control the evolution of the injectivity radius of a hyperbolic metric as we move in a horizontal direction at a given Weil-Petersson speed. The principal goal is to state and prove Lemma \ref{rootinjevol2} below. 
However, one of the ingredients, Lemma \ref{yaba_lemma}, will have many external applications. This latter lemma gives the sharp consequence of $g(t)$ being horizontal, by exploiting elliptic regularity to get $C^k$ control on $\pt g$ in terms of the $L^2$ norm of $\pt g$ (i.e. in terms of the speed that $g$ is moving through Weil-Petersson space) with sharp dependency on the injectivity radius.

\begin{lemma}
\label{inj_Lip_lem}
Suppose $g(t)$ is a horizontal curve on a closed oriented surface $M$, or indeed any smooth family of hyperbolic metrics. Then for each $x\in M$, the function
\beq
\label{inj_fn}
t\mapsto \inj_{g(t)}(x)
\eeq
is locally Lipschitz.
\end{lemma}

In particular, the function in \eqref{inj_fn} is differentiable for almost every $t\in I$. The following bound on its derivative is the main result of this section.

\begin{lemma}
\label{rootinjevol2}
Let $g(t)$ be a horizontal curve of hyperbolic metrics on a closed oriented surface $M$, 
for $t$ in some interval $I$.
Then there exists a constant $K_0<\infty$ depending only on the genus of $M$ such that
for any  $x\in M$ and almost all $t\in I$ (indeed, for every $t$ at which \eqref{inj_fn} is differentiable) we have 
\beq \label{est:inj-weaker-copy2}
\bigg|\ddt \left[\inj_{g(t)}(x)\right]^{\half}\bigg|\leq K_0
\norm{\partial_t g}_{L^2(M,g(t))}.
\eeq 
\end{lemma}

As alluded to above, if we see each $g(t)$ as representing a point in Teichm\"uller space, then the quantity $\norm{\partial_t g}_{L^2(M,g(t))}$ on the right-hand side of 
\eqref{est:inj-weaker-copy2} is, up to normalisation, the speed of $g(t)$ with respect to the Weil-Petersson metric. 
Of course, if one restricts this estimate to points $x$ of least injectivity radius, then one recovers a very weak form of the well-known lower bound on the Weil-Peterson distance to the boundary of Teichm\"uller space as found in the work of Wolpert \cite{wolpert07}, for example. We are most concerned with the estimate for general $x$ and estimates that are specific to the differential geometric viewpoint of Teichm\"uller theory.

\brmk
\label{si_rmk}
Recall that for any point $x$ in a complete hyperbolic surface $(M,g)$ other than the entire hyperbolic plane, there exists at least one unit speed geodesic $\si:[0,2\inj_{g}(x)]\to M$ starting and ending at $x$, that is homotopically nontrivial, and minimises length over all
homotopically nontrivial curves that start and end at $x$.
\ermk

\brmk
\label{inj_upper_bd}
We will repeatedly require that the injectivity radius at every point on our closed oriented hyperbolic surfaces is bounded above depending only on the genus. This is because by Gauss-Bonnet, the total area of the surface is determined by the genus alone.
\ermk

\begin{proof}[Proof of Lemma \ref{inj_Lip_lem}]
By replacing $I$ by an arbitrary smaller \emph{compact} interval, we may assume that $g(t)$ is a Lipschitz curve in the space of metrics equipped itself with the $C^0$ norm computed with respect to an arbitrary background metric $G$ on $M$. Thus there exists $C<\infty$ such that for all $a,b\in I$ we have
$\|g(a)-g(b)\|_G\leq C|a-b|$.
Also, on this smaller compact interval, the metrics $g(t)$ will be uniformly equivalent to $G$.
By Remark \ref{si_rmk}, we can pick homotopically nontrivial unit speed geodesics $\si_a:[0,2\inj_{g(a)}(x)]\to (M,g(a))$ and $\si_b:[0,2\inj_{g(b)}(x)]\to (M,g(b))$ starting and ending at $x$,
and
\beqa
2\inj_{g(a)}(x)&=L_{g(a)}(\si_a)\leq L_{g(a)}(\si_b)
=\int_0^{2\inj_{g(b)}(x)}\left[g(a) (\si_b',\si_b')\right]^\half\\
&\leq \half\int_0^{2\inj_{g(b)}(x)}\left(g(a) (\si_b',\si_b')+1\right)
\eeqa
because $x^\half\leq \half(x+1)$ for $x>0$.
Using $g(b) (\si_b',\si_b')=1$, we find
\beqa
2\left(\inj_{g(a)}(x)-\inj_{g(b)}(x)\right)&\leq
\half\int_0^{2\inj_{g(b)}(x)}\left(g(a) (\si_b',\si_b')-1\right)\\
&\leq 
\half\int_0^{2\inj_{g(b)}(x)}\left(g(a)-g(b)\right) (\si_b',\si_b')\\
&\leq \inj_{g(b)}(x)\|g(a)-g(b)\|_{g(b)}\\
&\leq C \inj_{g(b)}(x)\|g(a)-g(b)\|_{G}\leq C|a-b|,
\eeqa
by Remark \ref{inj_upper_bd}.
By switching $a$ and $b$, we obtain the desired Lipschitz bound.
\end{proof}

Now that we have established that the injectivity radius is differentiable for almost every $t$, we give a first formula for its derivative. One should keep in mind Remark \ref{si_rmk}.

\begin{lemma}
\label{weaker_inj_lemma}
Let  $g(t)$ be a horizontal curve of hyperbolic metrics on a closed oriented surface $M$ containing a point $x$, 
and suppose the function $t\mapsto \inj_{g(t)}(x)$ is differentiable at $t_0\in I$.
Then for any unit speed geodesic $\si:[0,2\inj_{g(t_0)}(x)]\to M$ starting and ending at $x$, we have
\beq
\label{weaker_inj_formula}
\frac{d}{dt}\inj_{g(t)}(x)\bigg|_{t=t_0}=
\frac14\int_0^{2\inj_{g(t_0)}(x)}\pt g (\si',\si').
\eeq
\end{lemma}

\begin{proof}
By Remark \ref{si_rmk}, for all $t\in I$, we have 
$$\inj_{g(t)}(x)\leq \half L_{g(t)}(\si),$$
with equality at least for $t=t_0$.
Therefore at $t=t_0$ the derivatives of 
$\inj_{g(t)}(x)$ and $\half L_{g(t)}(\si)$ coincide, and thus the lemma follows immediately.
\end{proof}

In order to improve \eqref{weaker_inj_formula}, we need to
understand the regularity implied by the curve $g(t)$ being horizontal.

\begin{lemma}
\label{yaba_lemma}
Let $g(t)$ be a horizontal curve of hyperbolic metrics on a closed oriented surface $M$, for $t$ in some interval $I$, and suppose $k\in \N\union\{0\}$.
Then there exists a constant $C<\infty$ depending on $k$ and the genus of $M$ such that
for any  $x\in M$ and $t\in I$ we have 
\beq
\label{yaba2}
|\pt{g}|_{C^k(g(t))}(x)\leq C[\inj_{g(t)}(x)]^{-\half}\|\pt g\|_{L^2(M,g(t))}.
\eeq
\end{lemma}

We clarify here that for any tensor $\Om$ defined in a neighbourhood of $x\in M$, we are writing
\beq
\label{Ck_notation}
|\Om|_{C^k(g)}(x):=\sum_{l=0}^k |\grad_{g}^{l} \Om|_g(x),
\eeq
where $\grad_g$ is the Levi-Civita connection of $g$, or its extension.
If $\Om$ is defined over an open subset containing some general subset  $K\subset M$, then we write
\beq
\label{Ck_notation2}
\|\Om\|_{C^k(K,g)}:= \sup_K |\Om|_{C^k(g)}.
\eeq

\brmk
Although we do not need it here, the proof below establishes not just \eqref{yaba2} but also the fact that
\beq
|\pt{g}|_{C^k(g(t))}(x)\leq C[\inj_{g(t)}(x)]^{-\half}\|\pt g\|_{L^2(B_{g(t)}(x,1))}.
\eeq
\ermk

\begin{proof}[Proof of Lemma \ref{yaba_lemma}]
For any $t\in I$, and any $x\in M$, we can apply elliptic regularity to deduce that $\pt g$ can be controlled pointwise in $C^k$ in terms of its local $L^1$ norm, as we now make precise.

First, recall that a closed hyperbolic surface decomposes into a union of collars $\Col$, with central geodesics of length 
$\ell\leq 2\arsinh(1)$, and the complement of the collars, on which the injectivity radius is always larger than $\arsinh(1)$, as described in Lemmata \ref{lemma:collar} and \ref{basic_inj_control}. 

Set $\nu:=\arsinh(1)$, with the understanding that it will shortly be reduced.
Suppose first that $x\in M$ does not lie in any of the collars, and so $\inj_{g(t)}(x)> \nu$.
In this case, elliptic theory (using the fact that $\pt g$ is the real part of a \emph{holomorphic} object and hence that in appropriate coordinates the components of $\partial_t g$ are given by harmonic functions) 
tells us that 
$$|\pt {g}|_{C^k(g(t))}(x)\leq C\|
\pt {g}
\|_{L^1(B_{g(t)}(x,\nu))}
\leq C\|\pt g\|_{L^2(M,g(t))},$$
%
where $C$ depends only on $k$ and the genus (and $\nu$ once it changes). Thus \eqref{yaba2} is proved in this case.

We are therefore reduced to the case that $x$ lies in a collar. Here we can use Lemma \ref{lemma:inj-collar} from the appendix, which tells us that 
we can reduce $\nu$ to some smaller universal number so that 
$\inj_{g(t)}(x)\leq \nu$ not only implies that $x$ lies in a collar, but also that it lies at least a geodesic distance $1$ from the ends of the collar. 
Note that the argument above still proves \eqref{yaba2} on the now larger set $\{x\ :\ \inj_{g(t)}(x)> \nu\}$.

Suppose therefore that we are in the remaining case that $\inj_{g(t)}(x)\leq \nu$, with this smaller $\nu$.
In this case we will pull back $\pt{g}$ 
to the ball $\cb$ of radius $1$ in the universal cover (hyperbolic space), and perform the elliptic regularity theory there.

We now claim that the cover $\Ups:\cb\to B_{g(t)}(x,1)$
sends a finite number of points, bounded above by $C/\inj_{g(t)}(x)$, with $C$ universal, to each point in the image, where we recall that $B_{g(t)}(x,1)$ is fully contained in a collar. 
To prove this claim, suppose that there are $n\geq 2$ points in $\cb$ all mapping to the same point under $\Ups$. Then there must exist a curve $\si$ in $\cb$ connecting two of these points whose image under $\Ups$ wraps $n-1$ times round the collar, and whose length is less than $2$.
(We can take the shortest geodesic from one point to the centre of $\cb$, and then add the shortest geodesic from the origin to the other point.)
Appealing to \eqref{elementary} of Lemma \ref{lemma:inj-collar}, we find that $\rho(y)\geq \frac1e\rho(x)$ for all $y\in B_{g(t)}(x,1)$. Therefore, the curve $\si$ must have length
bounded by
$$(n-1)\frac{2\pi\rho(x)}{e} \leq L(\si)<2,$$
and we deduce that 
$$n-1<\frac{e}{\pi\rho(x)}\leq \frac{e}{\inj_{g(t)}(x)},$$
by \eqref{going_in_appendix} of Lemma \ref{lemma:inj-collar}.
Because we have assumed that $\inj_{g(t)}(x)\leq \nu\leq \arsinh(1)$, we see that
$$n\leq \frac{C}{\inj_{g(t)}(x)},$$
completing the claim.

Thus the elliptic theory applied in $\cb$ now gives us
only 
\beq
\label{elliptic_part}
|\pt{g}|_{C^k(g(t))}(x)\leq \frac{C}{\inj_{g(t)}(x)}\|
\pt g
\|_{L^1(B_{g(t)}(x,1))}.
\eeq
But we know that $B_{g(t)}(x,1)$ lies within the $(\pi e)\inj_{g(t)}(x)$-thin part of
the collar, see \eqref{est:compare-inj-rad}, which has area controlled by $C\inj_{g(t)}(x)$ (see \cite[(A.2)]{RTZ}) for universal
$C$.
Therefore, by Cauchy-Schwarz, we have
$$\|\pt g\|_{L^1(B_{g(t)}(x,1))}\leq C\inj_{g(t)}(x)^\half
\|\pt g\|_{L^2(B_{g(t)}(x,1))},$$
and we find that
$$
|\pt{g}|_{C^k(g(t))}(x)\leq C[\inj_{g(t)}(x)]^{-\half}\|\pt g\|_{L^2(B_{g(t)}(x,1))},
$$
which completes the proof.
\end{proof}
\brmk
The theory in \cite{RT3} implies that the exponent $-\half$ for $\inj_{g(t)}(x)$ in 
Lemma 
\ref{yaba_lemma} is optimal. 
\ermk

\begin{proof}[Proof of Lemma \ref{rootinjevol2}.]
Let $t_0$ be a time at which \eqref{inj_fn} is differentiable.
By Lemma \ref{weaker_inj_lemma}, we can write
\beq
\label{sunnyday2}
\left|\frac{d}{dt}\inj_{g(t)}(x)\bigg|_{t=t_0}\right|\leq
\frac14\int_0^{2\inj_{g(t_0)}(x)}|\pt g (\si',\si')|
\leq \half\inj_{g(t_0)}(x)\sup_\si|\pt g|_{g(t_0)}.
\eeq

Remark \ref{inj_upper_bd} tells us that $\inj_{g(t_0)}(x)$ is bounded above in terms of the genus. Therefore, so also is the length of $\si$, as we will need twice below.

Suppose first that $x\in M$ does not lie in any collar 
with $\ell\leq 2\arsinh(1)$
at time $t_0$.
The boundedness of the length of $\si$ ensures that it passes only a bounded distance into any collar. Thus, by \eqref{est:inj-by-d} of Lemma \ref{lemma:inj-collar}, and the fact that the injectivity radius off the collars is bounded below by $\arsinh(1)$ (by Lemma \ref{basic_inj_control}) we find that 
the injectivity radius is bounded below all along $\si$ by some $\de>0$ depending at most on the genus.
By Lemma \ref{yaba_lemma}, we then know that
$$\sup_{\si}|\pt{g}|_{g(t_0)}\leq C\|\pt g\|_{L^2(M,g(t_0))},$$
with $C$ depending only on the genus.
Inserting this estimate into \eqref{sunnyday2}, gives
$$\left|\frac{d}{dt}\inj_{g(t)}(x)\bigg|_{t=t_0}\right|
\leq C\inj_{g(t_0)}(x)\|\pt g\|_{L^2(M,g(t_0))},$$
and keeping in mind our upper bound for $\inj_{g(t_0)}(x)$ we deduce \eqref{est:inj-weaker-copy2} as desired, in this case.

On the other hand, suppose that $x\in M$ lies in some collar $\Col(\ell)$, with 
$\ell\leq 2\arsinh(1)$. 
Now the boundedness of the length of $\si$ can be combined with 
\eqref{est:compare-inj-rad} from
Lemma \ref{lemma:inj-collar} to tell us that the injectivity radius along $\si$ is bounded below by $\ep\inj_{g(t_0)}(x)$ for some  $\ep>0$  depending only on the genus.
Therefore, by Lemma \ref{yaba_lemma}, for all $y\in \si$ we have
$$|\pt{g}|_{g(t_0)}(y)\leq C[\inj_{g(t_0)}(x)]^{-\half}\|\pt g\|_{L^2(M,g(t_0))},$$
and inserting this into \eqref{sunnyday2} we again obtain \eqref{est:inj-weaker-copy2}.
\end{proof}

\section{Convergence of horizontal curves to noncompact hyperbolic metrics}
\label{horiz_curv_sect}

\newcommand{\zzinject}{\mu}

The main objective of this section is to give the proof of our main Theorem \ref{convergence_of_g2}. However, some of the supporting lemmata will be independently useful; for example we use Lemma \ref{lemma:Ck-horiz} as an important ingredient in \cite{global}. 

One of the assertions of Theorem \ref{convergence_of_g2} is the existence of a set $\cu$, and the theorem would imply that $\cu$ would satisfy
\beq
\label{Udef}
\cu=\{p\in M:\, \liminf_{t\upto T}\inj_{g(t)}(p)>0\}.
\eeq
In this section, we will take \eqref{Udef} as the \emph{definition} of $\cu$, and verify that it has the desired properties.

As we move along a horizontal curve, the injectivity radius changes, and therefore the $\de\thick$ and $\de\thin$ parts will evolve. Lemma \ref{rootinjevol2} allows us to keep track of how they are nested, as we explain in the following lemma, which will be required in the proof of Theorem \ref{convergence_of_g2}.

\begin{lemma}
\label{nested_cor}
Let $M$ be a closed oriented surface of genus $\ga\geq 2$,
and suppose $g(t)$ is a smooth horizontal curve in $\m_{-1}$, for $t\in [0,T)$, with finite length $\cl(0)<\infty$ as defined in \eqref{length_def}.
Define $\cu\subset M$ by \eqref{Udef}, and define for each $\zzinject\geq 0$ 
and $t\in [0,T)$ the 
subset
$$M^\zzinject(t):=\{p\in M\ :\ \inj_{g(t)}(p)>(K_0 \cl(t)+\zzinject)^2\},$$
where $K_0$ is from Lemma \ref{rootinjevol2}.
Then for $0\leq t_1\leq t_2<T$, and $\tilde\zzinject\in [0,\zzinject]$, we have
\beq
\label{nesting_statement}
M^\zzinject(t_1)\subset M^{\tilde\zzinject}(t_2).
\eeq
Moreover, 
we have 
\beq
\label{MFchar}
\cu=\bigcup_{t\in [0,T)}M^0(t).
\eeq
and even 
\beq
\label{MFchar2}
\cu=\bigcup_{\zzinject >0 ,t\in [0,T)}M^\zzinject(t).
\eeq
%
%
Meanwhile, in the case that $\zzinject$ is positive, we have that 
the $(K_0\cl(0)+\zzinject)^2$-thick part of $(M,g(0))$ is contained within the $\zzinject^2$-thick part of $(M,g(t))$ for $t\in [0,T)$. 
\end{lemma}

\begin{proof}[Proof of Lemma \ref{nested_cor}]
To see the nesting of the sets $M^\zzinject(t)$ claimed in \eqref{nesting_statement},
we note first that reducing $\zzinject$ can only increase the size of $M^\zzinject(t)$, so we may as well assume that $\tilde\zzinject=\zzinject$.

By definition, if $p\in M^\zzinject(t_1)$, then 
$\inj_{g(t_1)}(p)^\half>K_0 \cl(t_1)+\zzinject$.
By Lemma \ref{rootinjevol2}, 
we have
$$
\left[\inj_{g(t_1)}(p)\right]^{\half}-\left[\inj_{g(t_2)}(p)\right]^{\half}
\leq K_0
\int_{t_1}^{t_2}\norm{\partial_t g}_{L^2(M,g(t))}dt=K_0(\cl(t_1)-\cl(t_2)),$$
and hence
$$\inj_{g(t_2)}(p)^\half>K_0 \cl(t_2)+\zzinject$$
as required to establish that $p\in M^\zzinject(t_2)$.

To see \eqref{MFchar} and \eqref{MFchar2}, suppose first that $p\in M^\zzinject(\tilde t)$ for some 
$\tilde t\in [0,T)$ and some $\zzinject> 0$. 
By the first part of the lemma, for all $t\in [\tilde t,T)$ we have
$p\in M^\zzinject(t)$ and hence
$\inj_{g(t)}(p)>(K_0\cl(t)+\zzinject)^2\geq \zzinject^2$, and therefore
$p\in \cu$ as required.

If, more generally, we have $p\in M^0(\tilde t)$ for some 
$\tilde t\in [0,T)$, then we can choose $\zzinject>0$ small so that 
$p\in M^\zzinject(\tilde t)$ and the argument above applies to show that 
$p\in \cu$.

Conversely, if $p\in \cu$, then by definition of $\cu$, there must exist some small $\zzinject>0$ so that
$\inj_{g(t)}(p)>(2\zzinject)^2$ for all $t\in [0,T)$.
But then if we take $t\in [0,T)$ sufficiently close to $T$ so that
$K_0\cl(t)\leq \zzinject$, which is possible since $\cl(0)<\infty$, then we must have $p\in M^\zzinject(t)\subset M^0(t)$ as required
in \eqref{MFchar2} and \eqref{MFchar}.

For the final part of the lemma,
note that if $p$ lies in the $(K_0\cl(0)+\zzinject)^2$-thick part of $(M,g(0))$, then $p\in M^{\tilde \zzinject}(0)$ for any $\tilde \zzinject\in [0,\zzinject)$. Therefore, by \eqref{nesting_statement}, we also have 
$p\in M^{\tilde \zzinject}(t)$, and hence $\inj_{g(t)}(p)>\tilde\zzinject^2$, 
for all $t\in [0,T)$. 
By taking the limit $\tilde\zzinject\upto \zzinject$, we see that
$p$ lies in the $\zzinject^2$-thick part of $(M,g(t))$ as claimed.
\end{proof}

Consider a horizontal curve $g(t)$ for $t\in [0,T)$ that degenerates as $t\upto T$. Then however large we take 
$t_0\in [0,T)$, the metrics $g(t)$ for $t\in [t_0,T)$ will clearly not be globally comparable.
We now want to state and prove a lemma telling us that on a thick-enough part, the metrics \emph{are} comparable. In fact, we will argue that the metrics are not just comparable as bilinear forms, but are even close in $C^k$. Lemma \ref{yaba_lemma} from the previous section tells us that $\pt g$ is controlled in $C^k$ in terms of the Weil-Petersson speed, where the injectivity radius is not too small, which makes $C^k$ closeness seem reasonable. However, this is $C^k$ control with respect to the evolving metric $g(t)$, whereas we need $C^k$ control with respect to a fixed metric.

\begin{lemma} \label{lemma:Ck-horiz}
Let $M$ be a closed oriented surface of genus $\ga\geq 2$.
Suppose $g(t)$ is a smooth horizontal curve in $\m_{-1}$, 
for $t\in [0,T)$, 
with finite length. 
Let $\delta>0$ and suppose $t_0\in [0,T)$ is sufficiently close to $T$ so that $(2K_0\cl(t_0))^2\leq\delta$, where $K_0$ is from 
Lemma \ref{rootinjevol2}. Then for any 
\beq
\label{x_hyp}
x\in \bigcup _{\tilde t\in[t_0,T)}\delta\thick(M,g(\tilde t))
\eeq
and any $s,t\in[t_0,T)$ we have
\beq
\label{est:equiv-metric}
g(s)(x)\leq C_1\cdot g(t)(x),
\eeq
and
\beq \label{est:equiv-inj}
\inj_{g(s)}(x)\leq C_2\cdot \inj_{g(t)}(x)
\eeq
where $C_1\in[1,\infty)$ depends only on the genus of $M$, 
while $C_2\in[1,\infty) $ is a universal constant. 
Furthermore, for any $k\in \N$ and any $x$ as above, we have
\beq \label{est:bamberg1} \abs{\pt g(t)}_{C^k(g(s))}(x)\leq C\delta^{-1/2}\norm{\pt g(t)}_{L^2(M,g(t))} \text{ for every } s,t\in [t_0,T),
\eeq
where $C$ depends only on $k$ and the genus of $M$.
In particular
\beq \label{est:bamberg2} \abs{g(t_1)-g(t_2)}_{C^k(g(s))}(x)\leq C\delta^{-1/2}(\cl(t_1)-\cl(t_2)), \eeq for any 
$t_0\leq t_1\leq t_2<T$, $s\in[t_0,T)$.
\end{lemma}

Before proving Lemma \ref{lemma:Ck-horiz}, we need to consider the evolution of norms of tensors and their covariant derivatives as the underlying metric evolves. We use the notation from \eqref{Ck_notation}.

\begin{lemma}
\label{tensor_evolve_lem}
Suppose, on a manifold $M$, that $g(t)$ is a smooth one-parameter family of metrics for $t\in [t_1,t_2]$, $\Om$ is a fixed smooth tensor, $x\in M$ and $k\in\N\union\{0\}$. Then there exists $C<\infty$ depending on the order of $\Om$, the dimension of $M$, and $k$ such that
\beq
\label{qwerty}
\left|
\pl{}{t}\nabla^k\Om\right|_{g(t)}(x)
\leq
C|\Om|_{C^k(g(t))}(x)
|\pt g|_{C^k(g(t))}(x).
\eeq
Moreover, for any $s_1,s_2\in[t_1,t_2]$, we have
\beq
\label{asdfg}
|\Om|_{C^k(g(s_1))}(x)\leq
|\Om|_{C^k(g(s_2))}(x)
\exp\left[
C\int_{t_1}^{t_2}|\pt g|_{C^k(g(t))}(x)dt
\right].
\eeq
\end{lemma}

\begin{proof}[Proof of Lemma \ref{tensor_evolve_lem}]
Instead of a fixed tensor $\Om$, we start by considering a smooth one-parameter family of tensors $\om(t)$. A standard computation (see e.g. \cite[(2.3.3)]{RFnotes}, and \cite[\S 2.1]{RFnotes} for $*$-notation) tells us that
\beq
\label{case1}
\pl{}{t}\nabla\om=\nabla\pl{\om}{t}+\om * \nabla\pl{g}{t},
\eeq
and by induction, this can be extended to
\beq
\label{general_case}
\pl{}{t}\nabla^l\om=\nabla^l\pl{\om}{t}+\sum_{i=1}^l\nabla^{l-i}\om * \nabla^i\pl{g}{t},
\eeq
for $l\in\N$, where the inductive step follows by replacing $\om$ in \eqref{case1} by $\nabla^{l-1}\om$. For example,
\beqa
\pl{}{t}\nabla^2\om&=\pl{}{t}\nabla(\nabla\om)=\nabla(\pl{\nabla\om}{t})+\nabla\om * \nabla\pl{g}{t}\\
&=\nabla\left(\nabla\pl{\om}{t}+\om * \nabla\pl{g}{t}\right)+\nabla\om * \nabla\pl{g}{t}\\
&=\nabla^2\pl{\om}{t}+\nabla\om * \nabla\pl{g}{t}
+\om * \nabla^2\pl{g}{t}.
\eeqa
Taking norms of \eqref{general_case} in the case that $\om(t)=\Om$ is independent of $t$ gives \eqref{qwerty}.

Meanwhile, considering again a smooth  one-parameter family of tensors $\om(t)$, we can consider the evolution of its norm $|\om(t)|_{g(t)}$, which is a Lipschitz function of $t$. 
Keeping in mind that 
$\pl{}{t}|\om|_{g(t)}^2=\pl{}{t}(\om * \om)=\pl{\om}{t} * \om + \pl{g}{t}*\om*\om$, 
we see that where $t\mapsto |\om(t)|_{g(t)}$ is
differentiable, we have
$$\pl{}{t}|\om|_{g(t)}\leq C|\om|_{g(t)}|\pt g|_{g(t)}
+C\left|\pl\om{t}\right|_{g(t)}.$$
We can apply this with $\om=\nabla^l\Om$, using \eqref{qwerty}, to give
\beqa
\pl{}{t}|\nabla^l\Om|_{g(t)}&\leq C|\nabla^l\Om|_{g(t)}|\pt g|_{g(t)}
+C\left|\pl{}{t}\nabla^l\Om\right|_{g(t)}\\
&\leq
C|\Om|_{C^k(g(t))}|\pt g|_{C^k(g(t))}
\eeqa
for $l\leq k$, at $x$, and hence the Lipschitz function
$t\mapsto |\Om|_{C^k(g(t))}(x)$
satisfies
$$\pl{}{t}|\Om|_{C^k(g(t))}\leq
C|\Om|_{C^k(g(t))}|\pt g|_{C^k(g(t))},$$
for almost all $t$, which can be integrated to give \eqref{asdfg}.
\end{proof}

\begin{proof}[Proof of Lemma \ref{lemma:Ck-horiz}]
We begin by observing that we may assume that $\cl(t)>0$ for all $t\in [0,T)$. Indeed, the case that $\cl(0)=0$ is trivial, and if $\cl(0)>0$ but $\cl(t)=0$ for some $t\in (0,T)$, then it would be enough to prove the lemma with $T$ replaced by the smallest value of $t$ for which $\cl(t)=0$.

Our first task is to establish \eqref{est:equiv-metric} and \eqref{est:equiv-inj}. 
For this part of the lemma it suffices to prove the claims with the hypothesis \eqref{x_hyp} replaced by the stronger condition that $x$ satisfies $\inj_{ g(t_0)}(x)\geq \delta$.
In the remaining situations that $x$ only satisfies $\inj_{ g(\tilde t)}(x)\geq \delta$ for some $\tilde t\in (t_0,T)$, we can reach the desired conclusion by applying the restricted claim first to $g(t)$ for $t$ restricted to
$[\tilde t,T)$, and then to the horizontal curve $g(\tilde t-t)$ for $t\in[0,\tilde t-t_0)$.
The combination of these two applications gives the general result, after adjusting the constants $C_1$ and $C_2$. 

Applying the last part of Lemma \ref{nested_cor} with $\mu=K_0\cl(t_0)$, which is possible because $(2K_0\cl(t_0))^2\leq \de$,
we find that for every $t\in [t_0,T)$ we have
\beq 
\label{est:phoenix}
\inj_{g(t)}(x)\geq\left[K_0\cl(t_0)\right]^2.
\eeq
Combining \eqref{est:phoenix} with Lemma \ref{rootinjevol2} gives
\beq 
\left|\ddt \big[\log(\inj_{g(t)}(x))\big]\right|=2\inj_{g(t)}(x)^{-1/2}\left|\ddt \big[\inj_{g(t)}(x)\big]^{1/2}\right|\leq 
\frac{2}{\cl(t_0)}\norm{\pt g}_{L^2(M,g(t))}
\eeq
which, once integrated over time, yields 
\eqref{est:equiv-inj}.
We then combine  Lemma \ref{yaba_lemma} with 
\eqref{est:phoenix} and obtain that for $t\in [t_0,T)$ we have
$$\abs{\pt g(t)}_{g(t)}(x)\leq \frac{C}{\cl(t_0)}\cdot \norm{\pt g(t)}_{L^2(M,g(t))},$$
where $C$ depends only on the genus of $M$.
Given any $t_1,t_2\in [t_0,T)$ we can thus 
estimate
\beq 
g(t_1)(x)\leq e^{\int_{t_0}^T\abs{\pt g(t)}_{g(t)}(x) dt}\cdot g(t_2)(x)\leq e^{C} g(t_2)(x)
\eeq 
which gives \eqref{est:equiv-metric}.

Having thus established \eqref{est:equiv-metric} and  \eqref{est:equiv-inj} for all points $x$ that satisfy the hypothesis \eqref{x_hyp}
we now turn to the proofs of \eqref{est:bamberg1} and \eqref{est:bamberg2}.

Because $\inj_{g(t)}(x)\geq C_2^{-1} \inj_{g(\tilde t)}(x)
\geq C_2^{-1} \de$, by \eqref{est:equiv-inj},
we can reduce \eqref{yaba2} of Lemma \ref{yaba_lemma} to 
\beq
\label{elliptic_part2}
|\pt{g}|_{C^k(g(t))}(x)\leq C
\de^{-\half}
\|\pt g\|_{L^2(M,g(t))}, \qquad \text{for }t\in[t_0,T),
\eeq
which, once combined with \eqref{asdfg} of Lemma \ref{tensor_evolve_lem}, tells us that for $s_1,s_2\in [t_0,T)$ we have
\beqa
|\Om|_{C^k(g(s_1))}(x)&\leq
|\Om|_{C^k(g(s_2))}(x)
\exp\left[
C
\de^{-\half}
\cl(t_0)
\right]\\
&\leq C |\Om|_{C^k(g(s_2))}(x),
\eeqa
because $(2K_0\cl(t_0))^2\leq \de$, where the constant $C$ depends only on $k$, the genus of $M$ and the order of the arbitrary tensor $\Om$.
Applying this in the case that $\Om=\pt{g}(t)$, for some fixed $t\in [t_0,T)$, and returning again to \eqref{elliptic_part2}, we  find that for every $s\in[t_0,T)$
$$|\pt{g}(t)|_{C^k(g(s))}(x)
\leq C
|\pt{g}(t)|_{C^k(g(t))}(x)
\leq C \delta^{-1/2}
\|\pt g(t)\|_{L^2(M,g(t))},$$
as claimed in \eqref{est:bamberg1}.
The final claim \eqref{est:bamberg2} then immediately follows by integrating \eqref{est:bamberg1} over time.
\end{proof}

When proving Theorem \ref{convergence_of_g2}, the following lemma will be useful in order to prove the completeness of the limit $h$.

\begin{lemma}
\label{thickthindist}
For all $\beta >0$ and $Q<\infty$, there exists $\ep>0$ depending on $\beta $ and $Q$ such that if $(M,g)$ is any closed oriented hyperbolic surface, we have 
$$\dist_g(\beta \thick(M,g),\ep\thin(M,g))>Q$$
whenever these sets are nonempty.
\end{lemma}

\begin{proof}[Proof of Lemma \ref{thickthindist}]
For the given $\beta $, being in the $\beta $-thick part forces us not to be too far down any collars, as we now explain. 
For our given $(M,g)$, consider the finitely many disjoint collars (as in Lemma \ref{thickthindist}) with 
$\ell\leq 2\arsinh(1)$. Assuming we ask that $\ep$ is smaller than $\arsinh(1)$, we can be sure that $\ep\thin(M,g)$ lies entirely within these collars, and indeed, by \eqref{est:inj-by-d} of Lemma \ref{lemma:inj-collar}, we can be sure that it lies at least a distance $-\log\sinh\ep$ from the boundary of the collars.

On the other hand, if $x_1\in\beta \thick(M,g)$ then either $x_1$ lies outside all these collars, or it does not lie too far within them. More precisely, again by \eqref{est:inj-by-d} of Lemma \ref{lemma:inj-collar}, the furthest that $x_1$ can lie from the boundary of a collar (within that collar) is $-\log\frac{\sinh\beta }{1+\sqrt2}$.

Thus, imposing also that $\ep<\beta $, we see that
$$\dist_g(\beta \thick(M,g),\ep\thin(M,g))\geq -\log\sinh\ep-\left(-\log\frac{\sinh\beta }{1+\sqrt2}\right),$$
and we can make the right-hand side larger than our given $Q$ by choosing $\ep>0$ sufficiently close to zero.
\end{proof}


\begin{proof}[Proof of Theorem \ref{convergence_of_g2}]
As explained at the start of 
the proof of Lemma \ref{lemma:Ck-horiz}, we may assume that $\cl(t)>0$ for all $t\in [0,T)$. 
As in the proof of Lemma \ref{nested_cor}, by applying Lemma \ref{rootinjevol2}, we know that
$$\left|
\left[\inj_{g(t_1)}(x)\right]^{\half}
-\left[\inj_{g(t_2)}(x)\right]^{\half}
\right|
\leq K_0\left(
\cl(t_1)-\cl(t_2)
\right)
$$
for $0\leq t_1\leq t_2<T$, and this implies that $\inj_{g(t)}$ converges in $C^0$ to some nonnegative limit $\hat\ci\in C^0(M)$ as $t\upto T$.
Moreover, by taking the limit $t_2\upto T$, we see that 
\beq
\label{unif_conv2}
\sup_{x\in M} \left\|[\inj_{g(t)}]^\half-{\hat\ci}^\half\right\|\leq K_0\cl(t)
\eeq
for all $t\in [0,T)$.
As remarked at the beginning of this section, we \emph{define}
$\cu$ by \eqref{Udef}. The uniform convergence we have just established shows that in fact 
\beq
\label{Udef2}
\cu=\{p\in M:\, \hat\ci(p)=\lim_{t\upto T}\inj_{g(t)}(p)>0\} 
\eeq
and that $\cu$ is open.
By the description of collars given in the appendix,
the set $\cu$ must be nonempty, i.e. the injectivity radius cannot converge to zero everywhere.

Given any $t_0\in [0,T)$ we let $\tilde\delta=\tilde\delta(t_0):=(2K_0\cl(t_0))^2>0$ and 
choose $\mu=\mu(t_0):=\half\tilde\delta^{1/2}>0$. Then 
$M^\mu(t_0)\subset \tilde\delta(t_0)\thick (M,g(t_0))$ 
(where $M^\mu$ is defined in Lemma \ref{nested_cor})
so 
\eqref{est:bamberg2} of Lemma \ref{lemma:Ck-horiz} (with $\de$ there equal to $\tilde\de$ here) yields
\beq
\label{zap}
\|g(t_1)-g(t)\|_{C^k(M^\zzinject(t_0),g(s))}\leq C
\zzinject^{-1}
\cl(t),
\eeq
for all $t_0\leq t\leq t_1<T$ and every $s\in [t_0,T)$, in particular for $s=t_0$.

Because of the completeness of $C^k(M^\zzinject(t_0),g(t_0))$, 
and the fact that $\cl(t)\downto 0$ as $t\upto T$, the estimate \eqref{zap} implies the existence of a tensor $h$ on $M^\zzinject(t_0)$ such that $g(t)\to h$ smoothly on $M^\zzinject(t_0)$ as $t\upto T$.

We now allow $t_0$ to increase to $T$, which forces $\tilde\delta(t_0)$ and thus also $\mu(t_0)$ to decrease towards $0$.
According to Lemma \ref{nested_cor}, by doing this, the nested sets $M^{\zzinject(t_0)}(t_0)$ exhaust the whole of $\cu$. Each time we increase $t_0$, we can extend the tensor 
$h$, and the convergence $g(t)\to h$, to the new $M^{\zzinject(t_0)}(t_0)$, and we end up with an extended $h$  and smooth local convergence $g(t)\to h$ as $t\upto T$ on the whole of $\cu$, as required.

For each $x\in \cu$ and $t_0<T$ sufficiently close to $T$ so that $x\in M^{\mu(t_0)}(t_0)$, and hence
$x\in \tilde\delta(t_0)\thick (M,g(t_0))$, 
we may apply \eqref{est:equiv-metric} of Lemma \ref{lemma:Ck-horiz} (with $\de$ there equal to $\tilde\de$ here) and take the limit $t\upto T$ to
see that 
$$g(t_0)(x)\leq C_1 h(x).$$
Therefore $h$ is nondegenerate, and is thus itself a metric.
Clearly, as a local smooth limit of hyperbolic metrics with uniformly bounded volume, $h$ must be hyperbolic and of finite volume.

The next step is to establish that $(\cu,h)$ is complete.
If $\cu=M$, then this is clear, so assume for the moment that $M\backslash \cu$ is nonempty.
In this case, for $\ep>0$, we define the nonempty set 
$$A_\ep:=\{p\in M:\, \hat\ci(p)=\lim_{t\upto T}\inj_{g(t)}(p)\leq\ep/2\}\supset 
M\backslash\cu,$$
and note that $M\backslash A_\ep\subset\subset \cu$ (because $\hat\ci$ is continuous) and that
$A_\ep$ shrinks to $M\backslash \cu$ as $\ep\downto 0$.

Pick $x_1\in \cu$. To show that $h$ is complete, it suffices to prove that for all $Q<\infty$, there exists $\ep>0$ such that $\dist_h(x_1,A_\ep)>Q$.
Since $x_1\in \cu$, we have $\beta :=\half\hat\ci(x_1)=\half\lim_{t\upto T}\inj_{g(t)}(x_1)>0$.
Now that we have both numbers $\beta $ and $Q$, Lemma \ref{thickthindist} gives us an $\ep>0$.
The uniform convergence $\inj_{g(t)}\to\hat\ci$ 
tells us that for $t<T$ sufficiently close to $T$, we have both that $x_1\in \beta \thick(M,g(t))$ and that
$A_\ep\subset \ep\thin(M,g(t))$. Therefore Lemma \ref{thickthindist} ensures that
$\dist_{g(t)}(x_1,A_\ep)>Q$. Taking the limit $t\upto T$ allows us to conclude that
$\dist_h(x_1,A_\ep)>Q$ and hence that $h$ is complete.

Next, we consider the geometry and conformal type of $(\cu,h)$.
Pick any $\hat x\in \cu$, and let $\hat M$ be the connected component of $\cu$ containing $\hat x$.
We have that $(M,g(t_n),\hat x)$ converges to $(\hat M,h,\hat x)$ in the Cheeger-Gromov sense, for any $t_n\upto T$ (where we take the diffeomorphisms in that notion of convergence to be restrictions of the inclusion map $\hat M \to M$). But the Deligne-Mumford-type theorem \ref{Mumford} ensures that the Cheeger-Gromov limit is a closed Riemann surface with finitely many punctures, equipped with a complete hyperbolic metric with cusp ends.
Indeed, when we analyse (a subsequence of) $(M,g(t_n))$ with Theorem \ref{Mumford}, we find that $\ka\in\{0,\ldots,3(\ga-1)\}$ collars degenerate, that there are a finite number $m$  components of $\cu$, and that there is a total of $2\ka$ punctures of the corresponding closed Riemann surfaces. Moreover, we find that if $\ka\geq 1$ (i.e. if $\cu$ is not the whole of $M$) then the genus of each Riemann surface is strictly less than that of $M$.

To bound $m$, observe that by Gauss-Bonnet, the Euler characteristic of each component of $(\cu,h)$ is no higher than $-1$ since it supports a hyperbolic metric. Therefore the total Euler characteristic is no higher than $-m$ and we see that $m\leq 2(\ga-1)$. (As a side remark, we have equality here if 
$\ka$ is as large as it can be, and each component of $(\cu,h)$ is conformally a 3-times punctured sphere. As another side remark, we must have $\ka\geq m-1$, the minimum number of collars required to connect the $m$ components together.)

On the other hand, to see that $M\backslash \cu$ has precisely $\ka$ connected components, we observe that by Lemma \ref{nested_cor}, as $t\in [0,T)$ increases, the closed sets $M\backslash M^0(t)$ shrink. By Lemmata \ref{lemma:collar} and \ref{basic_inj_control}, for $t\in [0,T)$ sufficiently close to $T$, this set will have exactly $\ka$ components, with one in each of the $\ka$ degenerating collars. Therefore, the intersection of these nested sets, which by Lemma \ref{nested_cor} is precisely $M\backslash \cu$, will also have precisely $\ka$ components.


Next, we turn to the assertion in the theorem that $\inj_{g(t)}$ converges to $\ci$ uniformly, with the estimate \eqref{unif_conv}. 
By \eqref{unif_conv2}, this amounts to proving the claim that 
$\inj_{g(t)}$ converges pointwise to
$\inj_h$ as $t\upto T$ on $\cu$, i.e. that $\ci=\hat\ci$. 

To prove the claim, consider the set $K$ of points in $(\cu, h)$ a distance no more than $2\inj_h(x)+1$ from some $x\in\cu$. Since $K$ is compact, we have $g(t)\to h$ in $C^k(K)$ for every $k\in\N$ as $t\upto T$. We can thus establish the claim by mimicking  the proof of Lemma 
\ref{inj_Lip_lem}, since all curves considered there will lie within $K$.

We now turn to the estimates \eqref{h_est_lem} and \eqref{h_est_lem_h}.
Let $\de$ and $t_0$ be as in the theorem, with  $\tilde \de$ corresponding to $t_0$ as above, so that $\de>\tilde\de=\tilde\de(t_0)$.
Thus we can apply Lemma \ref{lemma:Ck-horiz} (with $\de$ there now equal to $\de$ here). 
By \eqref{est:equiv-inj} (for example) we see that 
$\de\thick(M,g(s))\subset \cu$ 
for all $s\in [t_0,T)$.
Passing to the limit $t_2\upto T$ in \eqref{est:bamberg2} yields 
\eqref{h_est_lem} immediately.

Using the claim above that $\lim_{t\upto T}\inj_{g(t)}(x)=\inj_h(x)$ 
for each $x\in \cu$, we see that for every point 
$x\in \cu$ with 
$\inj_h(x)\geq \delta$, and any $\hat\de\in [\tilde \de,\de)$, we also have
$$x\in \bigcup _{\tilde t\in[t_0,T)}\hat\delta\thick(M,g(\tilde t)),$$
and \eqref{est:bamberg2} of Lemma \ref{lemma:Ck-horiz} implies that
\beq
\abs{g(t_1)-g(t_2)}_{C^k(g(s))}(x)\leq C\hat\delta^{-1/2}(\cl(t_1)-\cl(t_2)), 
\eeq 
for any 
$t_0\leq t_1\leq t_2<T$, $s\in[t_0,T)$.
We can then take the limits $s\upto T$, $\hat\de\upto\de$, and $t_2\upto T$ to obtain
\beq
\abs{g(t_1)-h}_{C^k(h)}(x)\leq C\delta^{-1/2}\cl(t_1), 
\eeq 
with $C$ depending only on $k$ and the genus of $M$, as required for 
\eqref{h_est_lem_h}.
\end{proof}

\brmk \label{rem:h}
Having proved Theorem \ref{convergence_of_g2}, and in particular the smooth local convergence $g(t)\to h$ on $\cu$ and the convergence of the injectivity radii, we can return to Lemma \ref{lemma:Ck-horiz} to record its consequences for the limit metric $h$.
We see, precisely, that if $M$, $g(t)$, $\de$ and $t_0$ are as in Lemma \ref{lemma:Ck-horiz}, and $h$ and $\cu$ are as in Theorem \ref{convergence_of_g2}, then for all $t\in [t_0,T)$ and any $x\in M$ satisfying  $\inj_{g(\tilde t)}(x)\geq \de$ for some 
$\tilde t\in [t_0,T)$, and in particular for any $x\in \cu$ satisfying 
$\inj_{h}(x)> \de$, we have the estimates
\beq\label{est:equiv-h}
C_1^{-1}\cdot h(x)\leq g(t)(x)\leq C_1 \cdot h(x) \quad \text{ and }\quad C_2^{-1}\cdot \inj_h(x)\leq \inj_{g(t)}\leq C_2\cdot \inj_h(x)
\eeq
and 
\beq 
\label{est:Ck-with-h}
\abs{\pt g(t)}_{C^k(h)}(x)\leq C\delta^{-\half}\norm{\pt g(t)}_{L^2(M,g(t))},
\eeq
with the constants $C,C_1,C_2>0$ obtained in Lemma \ref{lemma:Ck-horiz}. 
\ermk

\begin{rmk}
\label{rmk:piecewise}
Although the results of this paper are presented for simplicity
under the assumption that the horizontal curves are smooth,
they extend easily to, for example, continuous \textit{piecewise smooth} horizontal curves, as obtained
when analysing the metric component of
a solution of Teichm\"uller harmonic map flow that has singularities
caused by the bubbling off of harmonic spheres, as in \cite{Rexistence} and \cite{global}.
We may apply the results of Section \ref{inj_rad_sect} 
on each time interval over which the curve is smooth and hence analyse the whole curve
precisely as done in Section \ref{horiz_curv_sect}. 
In particular,
Theorem \ref{convergence_of_g2} and 
Lemmata \ref{nested_cor} and \ref{lemma:Ck-horiz}
extend without change to this more general setting.
\end{rmk}


\appendix
\section{Appendix}
We collect a few fundamental properties of hyperbolic metrics.

\begin{lemma}[The collar lemma, Keen-Randol \cite{randol}] \label{lemma:collar}
Let $(M,g)$ be a closed orientable hyperbolic surface and let $\si$ be a simple closed geodesic of length $\ell$. Then there is a neighbourhood around $\si$, a so-called collar, which is isometric to the 
cylinder 
$\Col(\ell):=(-X(\ell),X(\ell))\times S^1$
equipped with the metric $\rho^2(s)(ds^2+d\theta^2)$ where 
$$\rho(s)=\frac{\ell}{2\pi \cos(\frac{\ell s}{2\pi})} 
\qquad\text{ and }\qquad  
X(\ell)=\frac{2\pi}{\ell}\left(\frac\pi2-\arctan\left(\sinh\left(\frac{\ell}{2}\right)\right) \right).$$ 
The geodesic $\si$  corresponds to the circle 
$\{s=0\}\subset \Col(\ell)$. 
\end{lemma}

We will need to understand the injectivity radius within a hyperbolic surface, both on and off the collar regions.

\begin{lemma}[{Special case of \cite[Theorem 4.1.6]{Buser}}]
\label{basic_inj_control}
Let $(M,g)$ be a closed orientable hyperbolic surface. If $x\in M$ does not lie in any collar region for which 
$\ell\leq2\arsinh(1)$, then $\inj_g(x)>\arsinh(1)$.
On the other hand, if $x$ does lie in a collar $\Col$ for which 
$\ell\leq 2\arsinh(1)$, then 
\beq
\label{inj_d_formula}
\sinh(\inj_g(x))=\cosh(\ell/2)\cosh d(x)-\sinh d(x),
\eeq
where $d(x):=\dist_g(x,\partial \Col)$ denotes the geodesic distance to an end of the collar.
\end{lemma}
The largest that $d(x)$ can be is when $x$ lies at the centre of the collar, in which case $\inj_g(x)=\ell/2$, and we have equality in the inequality
\beq
\label{extreme_case}
\sinh(\ell/2)\sinh(d(x))\leq 1
\eeq
that follows from \eqref{inj_d_formula}.


For points contained in such collars we furthermore use:
\begin{lemma} \label{lemma:inj-collar}
Let $(M,g)$ be a closed orientable hyperbolic surface,
and let $x\in M$ be any point that is contained in a collar $\Col$ with central geodesic of length 
$\ell\leq 2\arsinh(1)$. 
Then 
\beq \label{est:inj-by-d}
\arsinh(e^{-d(x)})\leq \inj_g(x)\leq \arsinh((1+\sqrt{2})e^{-d(x)})
\eeq
where $d(x):=\dist_g(x,\partial \Col)$ as before. Furthermore,
for any $r>0$ 
\beq
\label{elementary}
\rho(x)e^{-r}\leq \rho(y)\leq \rho(x)e^r
\qquad\text{ for every } 
y\in B_g(x,r)\cap \Col.
\eeq
Moreover, $\rho$ is comparable with the injectivity radius,
\beq
\label{going_in_appendix}
\rho(y)\leq \inj_{g}(y)\leq \pi\rho(y)\qquad\text{ for all }y\in\Col,
\eeq
and so 
\beq \label{est:compare-inj-rad}
\inj_g(x)\cdot (\pi\cdot e^r)^{-1}\leq \inj_g(y)\leq \inj_g(x) \cdot (\pi\cdot e^r) \qquad\text{ for every } 
y\in B_g(x,r)\cap \Col.
\eeq
\end{lemma}

\begin{proof}[Proof of Lemma \ref{lemma:inj-collar}]
By \eqref{inj_d_formula}, we have
$$\sinh(\inj_g(x))=\cosh(\ell/2)\cosh d(x)-\sinh d(x)
\geq \cosh d(x)-\sinh d(x)=e^{-d(x)},$$
which is the first part of \eqref{est:inj-by-d}.
We split the second part of \eqref{est:inj-by-d} into two cases.

{\bf Case 1:} If $d(x)\leq \arsinh(1)$, 
equivalently $\sqrt 2\sinh d(x)\leq \cosh d(x)$, then we use the bound 
$\ell\leq2\arsinh(1)$,
equivalently $\cosh(\ell/2)\leq \sqrt{2}$,
in \eqref{inj_d_formula} to find that
\beqa
\sinh(\inj_g(x))
&\leq \sqrt{2}\cosh d(x)-\sinh d(x)\\
&\leq \left(\sqrt{2}\cosh d(x)-\sinh d(x)\right)
+\left(\cosh d(x)-\sqrt 2\sinh d(x)\right)\\
&=(1+\sqrt2 )e^{-d(x)}.
\eeqa
{\bf Case 2:} If $d(x)> \arsinh(1)$, equivalently $\la:=\frac{\cosh d(x)}{\sinh d(x)}\in [1,\sqrt 2)$, then we use the bound \eqref{extreme_case} for $\ell$, equivalently $\cosh(\ell/2)\leq \la$, in \eqref{inj_d_formula} to find that
\beqa
\sinh(\inj_g(x))-(1+\sqrt2 )e^{-d(x)}
&\leq 
\la\cosh d(x)-\sinh d(x)-(1+\sqrt2 )(\cosh d(x)-\sinh d(x))\\
&=\sinh d(x)(\la-1)(\la-\sqrt 2)\leq 0,
\eeqa
which completes the proof of \eqref{est:inj-by-d}.

Next, by simple computation (cf. \cite[(A.5)]{RT3}), we know that 
$$\bigg|\frac{d}{ds}\log\rho(s)\bigg|\leq \rho(s),$$
and by integrating over $s$, we find that 
for all $y\in B_g(x,r)\cap \Col$ we obtain \eqref{elementary}.

Finally, the elementary second inequality of \eqref{going_in_appendix} is from \cite[(A.8)]{RT3} while the first inequality is from \cite[(A.9)]{RT3}.
\end{proof}

We also recall the following differential geometric version of the Deligne-Mumford compactness theorem.

\begin{thm} {\rm (Deligne-Mumford compactness, cf. \cite{Hu}.)}  \label{Mumford}
Let $(M,g_{n})$ be a sequence of closed oriented hyperbolic Riemann surfaces of genus $\gamma\geq2$. 
Then, after the selection of a subsequence,
$(M,g_{n})$ converges to a complete hyperbolic Riemannian surface in the following sense.
There exist $\ka\in\{0,\ldots,3(\ga-1)\}$, a collection of pairwise disjoint simple closed curves
$\mathscr{E}=\{\sigma^{j}, j=1,...,\ka\}$ on $M$, a complete hyperbolic metric $g_\infty$ on the surface
$\Si:=M\backslash \union_{j=1}^\ka\si^j$, and a sequence of diffeomorphisms $F_n:M\to M$ such that the following is true.

First, the surface $(\Si,g_\infty)$ is conformal to the disjoint union of a finite collection of closed Riemann surfaces $\{M_i\}$, with a total of $2\ka$ punctures. 
If $\ka\geq 1$, then the  genus of each $M_i$ is strictly less than that of $M$.
A neighbourhood of each of these punctures is isometric to a hyperbolic cusp. 
Second, for each $n\in\N$ and $j=1,...,\ka$, the simple closed curves 
$\sigma^{j}_{n}:=F_n\circ\sigma^j$ are geodesics on $(M,g_{n})$
with lengths
$\ell_{n}^{j}:=\ell(\sigma_{n}^{j})\to 0$ as $n \rightarrow \infty$,
such that
the restricted diffeomorphisms
$f_n=F_n|_\Si:\Si\rightarrow M\setminus \cup_{j=1}^\ka\sigma_{n}^{j} $ satisfy
$$(f_n)^{*}g_{n} \rightarrow g_\infty \text{ in } C_{loc}^{\infty}(\Sigma).$$
For sufficiently small $\de>0$, while the $\de\thin$ part of $(M,g_n)$ will lie within the union of the collars $\union_{j=1}^\ka\Col_n^j$ around the geodesics $\si_n^j$, the preimage under $F_n$ of the $\de\thick$ part of $(M,g_n)$  remains within an $n$-independent subset $K_\de\subset\subset \Si$.
\end{thm}



{\sc MR: 
Mathematical Institute, University of Oxford, Oxford, OX2 6GG, UK}

{\sc PT: Mathematics Institute, University of Warwick, Coventry,
CV4 7AL, UK}

\end{document}